\documentclass[letterpaper,12pt,onecolumn]{article}

\usepackage[]{amsmath,amssymb,amsthm, url,fullpage}

\newtheorem{theorem}{Theorem}[section]
\newtheorem{lemma}[theorem]{Lemma}
\newtheorem{claim}[theorem]{Claim}

\def\disc{{\mathrm{disc}}}
\def\herdisc{{\mathrm{herdisc}}}
\def\E{{\mathbb{E}}}
\def\Z{{\mathbb{Z}}}
\def\L{{\mathcal{L}}}
\def\girth{{\mathrm{girth}}}

\begin{document}
% \pagestyle{headings}

%\begin{titlepage}
\title{On the Beck-Fiala Conjecture for Random Set Systems %\thanks{
  %Work on this paper has been supported by NSF under grant CCF-12-16689. }
}
\author{Esther Ezra
  \footnote{
    School of Mathematics,
    Georgia Institute of Technology, Atlanta, Georgia 30332, USA.
    \texttt{eezra3@math.gatech.edu}.
  }
  \and
  Shachar Lovett
  \footnote{Computer Science and Engineering, University of California, San Diego.
    \texttt{slovett@cse.ucsd.edu}.
    Research supported by an NSF CAREER award and a Sloan fellowship.
  }
  %\thanks{%
  %  Courant Institute of Mathematical Sciences,
  %  New York University, New York, NY 10012, USA;
  %  \textsl{esther@courant.nyu.edu}.
    %}
}

% \date{}
\maketitle

\begin{abstract}
Motivated by the Beck-Fiala conjecture, we study discrepancy bounds for random sparse set systems.
Concretely, these are set systems $(X,\Sigma)$, where each element $x \in X$ lies in $t$
randomly selected sets of $\Sigma$, where $t$ is an integer parameter.
We provide new bounds in two regimes of parameters.
We show that when $|\Sigma| \ge |X|$ the hereditary discrepancy of $(X,\Sigma)$ is with high probability $O(\sqrt{t \log t})$;
and when $|X| \gg |\Sigma|^t$ the hereditary discrepancy of $(X,\Sigma)$ is with high probability $O(1)$.
The first bound combines the Lov{\'a}sz Local Lemma with a new argument based on partial matchings;
the second follows from an analysis of the lattice spanned by sparse vectors.
\end{abstract}
%\end{titlepage}

%keywords: Geometric Discrepancy, Primal shatter function, Partial coloring and entropy method, \delta-packing, Relative approximations and epsilon-nets.

%\vspace{-2ex}
%\section{Introduction}
%\label{sec:intro}
%\vspace{-2ex}

\section{Introduction}
Let $(X, \Sigma)$ be a finite set system, with $X$ a finite set and $\Sigma$ a collection of subsets of $X$.
A \emph{two-coloring} of $X$ is a mapping $\chi : X \rightarrow \{-1 ,+1\}$.
For a subset $S \in \Sigma$ we define $\chi(S) := \sum_{x \in S} \chi(x)$.
The \emph{discrepancy} of $\Sigma$ is defined as
$$
\disc(\Sigma) := \min_{\chi} \max_{S \in \Sigma} |\chi(S)|  .
$$
In other words, the discrepancy of the set system $(X,\Sigma)$ is the minimum over all colorings $\chi$ of the
largest deviation from an even split, over all subsets in $\Sigma$. For background on discrepancy theory, we refer the
reader to the books of Chazelle~\cite{Chaz-01} and
%see the excellent book of
Matou\v{s}ek~\cite{matousek2009geometric}.

In this paper, our interest is in the discrepancy of sparse set systems.
The set system $(X,\Sigma)$ is said to be $t$-sparse if any element $x \in X$ belongs to at most $t$
sets $S \in \Sigma$.
A well-known result of Beck and Fiala~\cite{beck1981integer} is that sparse set systems have discrepancy bounded only in terms of their sparsity.

\begin{theorem}[\cite{beck1981integer}]
  \label{thm:bf}
  If $(X,\Sigma)$ is $t$-sparse then $\disc(\Sigma) \le 2t-1$.
\end{theorem}

Beck and Fiala conjectured that in fact, the bound can be improved to $O(\sqrt{t})$,
analogous to Spencer's theorem for non-sparse set systems~\cite{spencer1985six}.
This is a long standing open problem in discrepancy theory. The best result to date is by Banaszczyk~\cite{banaszczyk1998balancing}.

\begin{theorem}[\cite{banaszczyk1998balancing}]
If $(X,\Sigma)$ is $t$-sparse with $|X|=n$ then $\disc(\Sigma) \le O(\sqrt{t \log n})$.
\end{theorem}

\paragraph{Our results.}

In this paper, we study random sparse set systems.
To sample a random $t$-sparse set system $(X,\Sigma)$ with $|X|=n,|\Sigma|=m$, for each $x \in X$
choose uniformly and independently a subset $T_x \subset [m]$ of size $|T_x|=t$. Then set $S_i=\{x \in X: i \in T_x\}$ and $\Sigma=\{S_1,\ldots,S_m\}$. Letting $\E[\cdot]$ denote expectation, our main quantity of interest is $\E[\disc(\Sigma)]$. We show that when $m \ge n$, this is close to the conjectured bound of Beck and Fiala.
Specifically, we show $\E[\disc(\Sigma)] = O(\sqrt{t\log{t}})$. In particular, the bound does not depend on $n$.

In fact, we obtain such bound for the hereditary discrepancy of the set system.
For $Y \subset X$ let $\Sigma|_Y=\{S \cap Y: S \in \Sigma\}$ be the set system restricted to $Y$. The
hereditary discrepancy of a set system $(X,\Sigma)$ is defined as
$$
\herdisc(\Sigma) = \max_{Y \subset X} \disc(\Sigma|_Y).
$$
Our main result is the following.

\begin{theorem}
\label{thm:herdisc}
Assume $m \ge n \ge t$. Let $(X,\Sigma)$ be a random $t$-sparse set system with $|X|=n, |\Sigma|=m$.
Then
$$
\E[\disc(\Sigma)] \le \E[\herdisc(\Sigma)] \le O(\sqrt{t\log{t}}) .
$$
In fact, the bound holds with probability $1-\exp(-\Omega(t))$.
\end{theorem}

We note that our technique can be extended to the case where $m \ge cn$ for any absolute constant $c>0$,
but fails whenever $m \ll n$. The main reason is that in this regime, most sets are large.
Nevertheless, when $n$ is considerably larger than $m$, we use a different approach and show that the discrepancy is small
in this case as well. Specifically, when $n$ is somewhat larger than ${m \choose t}$ we show that the discrepancy
is only $O(1)$.

\begin{theorem}
\label{thm:disc_many_cols}
Fix $m \ge t$ and let $N = {m \choose t}$.
Assume that $n \ge \Omega(N \log{N})$.
Let $(X,\Sigma)$ be a random $t$-sparse set system with $|X|=n, |\Sigma|=m$.
Then
$$
\E[\disc(\Sigma)] = O(1) .
$$
In fact, the bound holds with probability $1-N^{-\Omega(1)}$.
\end{theorem}

%We conjecture that our bound for $m \le n$ holds also in the full range of $m \ge n$ (although it is not tight whenever $m \gg {n \choose t}$, as we just saw).
%
%\begin{conjecture}
%Assume $m,n \ge t$. Let $(X,\Sigma)$ be a random $t$-sparse set system with $|X|=n, |\Sigma|=m$. Then
%$$
%\E[\herdisc(\Sigma)] \le O(\sqrt{t\log{t}}) .
%$$
%\end{conjecture}

To summarize, the work in this paper was motivated by the elusive Beck-Fiala conjecture. We considered a natural setting of random $t$-sparse set systems, 
and showed that in this case, in some regimes of parameters, the conjecture holds (with the bound of $O(\sqrt{t})$ replaced by the slightly weaker
bound of $O(\sqrt{t \log t})$ in our first result). We hope that the techniques developed in this work will be useful for the study of random sparse set systems in the
full spectrum of parameters, as well as for the original Beck-Fiala conjecture.

\section{Preliminaries and Proof Overview}

\paragraph{The Lov{\'a}sz Local Lemma.}
The Lov{\'a}sz Local Lemma~\cite{erdos1975problems} is a powerful probabilistic tool. In this paper we only need its symmetric version.

\begin{theorem}
\label{thm:LLL}
Let $E_1, E_2,..., E_k$ be a series of events such that each event occurs with probability at most $p$ and such that each event is independent of all the other events except for at most $d$ of them. If $ep(d+1) \le 1$ then $\Pr[\wedge_{i=1}^m \overline{E_i}]>0$.
\end{theorem}

\paragraph{Tail bounds.}
In our analysis we exploit a few standard tail bounds for the sum of independent random variables
(Chernoff-Hoeffding bounds, see, e.g.,~\cite{AS-00}).

\begin{lemma}[Tail bounds for additive error]
\label{lemma:tail1}
Let $Z_1,\ldots,Z_k \in \{-1,1\}$ be independent random variables and let $Z=Z_1+\ldots+Z_k$. Then for any $\lambda>0$
$$
\Pr\left[|Z- \E[Z]| \ge \lambda \sqrt{k}\right] \le 2 \exp(-2 \lambda^2).
$$
\end{lemma}

\begin{lemma}[Tail bounds for multiplicative errors]
\label{lemma:tail2}
Let $Z_1,\ldots,Z_k \in \{0,1\}$ be independent random variables and let $Z=Z_1+\ldots+Z_k$. Then for any $\lambda>0$
$$
\Pr\left[Z \ge (1+\lambda) \E[Z]\right] \le \exp(-\lambda^2/3 \cdot \E[Z]).
$$
\end{lemma}

\subsection{Proof Overview for Theorem~\ref{thm:herdisc} }

We next present an overview of our proof for Theorem~\ref{thm:herdisc}.
For simplicity of exposition, we present the overview only for the derivation of the discrepancy bound.
In Section~\ref{sec:analysis} we present the actual analysis and show a bound on the hereditary discrepancy.

First, we classify each set as being either ``small'' if its cardinality is $O(t)$, or ``large'' otherwise.
%rows of $A$ according to their length. That is, a row $v$ in $A$ is \emph{long} if
%$\|v\| > 6t$, and it is \emph{short} otherwise.
%
Then we proceed in several steps:

\begin{itemize}
\item{(i)}~\textbf{Making large sets pairwise disjoint:}
Initially, we show that with high probability over the choice of the set system, it is possible to delete at most one element from each large set, such that
they become pairwise disjoint after the deletion. This property is proved in Lemma~\ref{lemma:decompose}.

\item{(ii)}~\textbf{Partial matching:}
%For each long row $v$, consider its non-zero elements (remaining after step (i)), by construction, these elements appear in distinct columns.
For each large set resulting after step (i), we pair its elements, leaving at most, say, two unpaired elements.
Since each pair appears in a unique set, this process results in a \emph{partial matching} $M=\{(a_1,b_1),\ldots,(a_k,b_k)\}$ on $X$.
We observe that as soon as we have such a matching, we can restrict the two-coloring function $\chi$ on $X$ to assign
\emph{alternating signs} on each pair of $M$. Since each large set $S$ has at most two unpaired elements, we immediately conclude that
$|\chi(S)| \le 2$.

\item{(iii)}~\textbf{Applying the Lov{\'a}sz Local Lemma on the small sets:}
We are thus left to handle the small sets.
In this case, we observe that a random coloring $\chi$, with alternating signs on $M$ as above\footnote{That is, each pair in $M$ is assigned $(+1,-1)$ or $(-1,+1)$ independently with probability $1/2$.}, satisfies with positive probability that
$|\chi(S)| \le O(\sqrt{t \log t})$ for all small sets $S \in \Sigma$.
This is a consequence of the Lov{\'a}sz Local Lemma, as each small set $S$ contains only $O(t)$ elements, and each of these elements
participates in $t$ sets of $\Sigma$. The fact that some of these elements appear in the partial matching implies that $S$ can ``influence''
(w.r.t. the random coloring $\chi$) at most $2|S| t = O(t^2)$ other small sets; see Section~\ref{sec:analysis} for the details.
\end{itemize}

We point out that as soon as we have a partial matching $M$ as above, we can ``neutralize'' the deviation that might be caused by the large sets,
and only need to keep the deviation, caused by the small sets, small. The latter is fairly standard to do, and so the main effort in the analysis
is to show that we can indeed make large sets disjoint as in step (i).

We note that our proof technique is constructive. Our arguments for steps (i) and (ii) (see Lemma~\ref{lemma:decompose}
and our charging scheme in Claim~\ref{claim:nocycle_decompose}) give an efficient algorithm to
find an element to delete in each large set, thereby making large sets disjoint, as well as build the partial matching,
or, alternatively, report (with small probability) that a partial matching of the above kind does not exist and halt.
In step (iii) we can apply the algorithmic Lov{\'a}sz Local Lemma of Moser and Tardos~\cite{moser2009constructive,moser2010constructive}, since the colors are assigned independently among the pairs in $M$ as well as the unpaired elements.
Thus, we obtain an expected polynomial time algorithm, which, with high probability over the choice of the set system,
constructs a coloring with discrepancy $O(\sqrt{t \log t})$.

\section{A Low Hereditary Discrepancy Bound: The Analysis}
\label{sec:analysis}

We now proceed with the proof of Theorem~\ref{thm:herdisc}.
We classify the sets in $\Sigma$ based on their size. A set $S \in \Sigma$ is said to be \emph{large} if $|S| \ge 6t$ and \emph{small} otherwise.
Note that as $m \ge n$, most sets in $\Sigma$ are small. Let $I=\{i: S_i \textrm{ is large}\}$ be a random variable capturing the indices of the large sets.
To construct a coloring, we proceed in several steps. First, we show that with high probability the large sets are nearly disjoint. We will assume throughout that $t$ is sufficiently large (concretely $t \ge 55$).

\begin{lemma}
\label{lemma:decompose}
Fix $t \ge 55$.
Let $E$ denote the following event: ``there exists a choice of $x_i \in S_i$ for $i \in I$ such that the sets $\{S_i \setminus \{x_i\}: i \in I\}$ are pairwise disjoint". Then
$\Pr[E] \ge 1-2^{-t}$.
\end{lemma}

We defer the proof of Lemma~\ref{lemma:decompose} to Section~\ref{sec:proof_lemma_decompose} and prove Theorem~\ref{thm:herdisc} based on it, in the remainder of this section.
Decompose
\begin{align*}
\E[\herdisc(\Sigma)] &= \E[\herdisc(\Sigma) | E] \Pr[E] + \E[\herdisc(\Sigma) | \overline{E}] \Pr[\overline{E}] \\
& \le \E[\herdisc(\Sigma) | E]+(2t-1) \Pr[\overline{E}]\\
& \le \E[\herdisc(\Sigma) | E]+1
\end{align*}
where we bounded $\E[\herdisc(\Sigma) | \overline{E}]$ by the Beck-Fiala theorem (Theorem~\ref{thm:bf}) which holds for any $t$-sparse set system, and bounded $\Pr[\overline{E}]$ by $2^{-t}$ according to Lemma~\ref{lemma:decompose}.
To conclude the proof we will show that when $E$ holds then $\herdisc(\Sigma) \le O(\sqrt{t \log t})$.
Thus, we assume from now on that the event $E$ holds. Fix a subset $Y \subset X$, where we will construct a two-coloring for $\Sigma'=\Sigma|_Y$ of low discrepancy.

Partition each $S_i \cap Y = A_i \cup B_i$ for $i \in I$, where $|A_i|$ is even, $|B_i| \le 2$ and the sets $\{A_i: i \in I\}$ are pairwise disjoint.
Partition each $A_i$ arbitrarily into $|A_i|/2$ pairs, and let $M$ be the union of these pairs.
That is, $M$ is a partial matching on $Y$ given by $M=\{(a_1,b_1),\ldots,(a_k,b_k)\}$ where $a_1,b_1,\ldots,a_k,b_k \in Y$ are distinct,
and each $A_i$ is a union of a subset of $M$, and each pair $a_j, b_j$ appears in a unique set $A_i$ due to the fact that these sets are pairwise
disjoint (they thus form a partition of $M$).
% \esther{why? $A_i$ is the union of all elements $a_i$, $b_i$ appearing in the pairs in $M$.}.
We say that a coloring $\chi:Y \to \{-1,+1\}$ is \emph{consistent with $M$} if $\chi(a_j)=-\chi(b_j)$ for all $j \in [k]$. Note that
if $S_i$ is a large set, then for any coloring $\chi$ consistent with $M$,
$$
|\chi(S_i \cap Y)|=|\chi(A_i)+\chi(B_i)|=|0+\chi(B_i)| \le |B_i| \le 2.
$$
Thus, we only need to minimize the discrepancy of $\chi$ over the small sets in $\Sigma$.
To do so, we choose $\chi$ uniformly from all two-colorings consistent with $M$. These are given
by choosing uniformly and independently $\chi(a_i) \in \{-1,+1\}$ for $i \in [k]$, setting $\chi(b_i)=-\chi(a_i)$ and choosing $\chi(x) \in \{-1,+1\}$ uniformly and independently
for all $x \notin \{a_1,b_1,\ldots,a_k,b_k\}$.

Let $S_i$ be a small set, that is $|S_i| \le 6t$. Let $E_i$ denote the event
$$
E_i := \left[|\chi(S_i \cap Y)| \ge c \sqrt{t \log t}\right] .
$$
Each pair $\{a_{j},b_{j}\}$ contained in $S_i$ contributes $0$ to the discrepancy, and all other elements obtain independent colors. Hence
$\chi(S_i)$ is the sum of $t' \le 6t$ independent signs. By Lemma~\ref{lemma:tail1}, for an appropriate constant $c$ we have
$$
\Pr[E_i] \le 1/100 t^2.
$$
We next claim that each event $E_i$ depends on at most $d=12 t^2$ other events $\{E_j: j \ne i\}$.
Indeed, let $S'_i = S_i \cup \{a_j: b_j \in S_i\} \cup \{b_j: a_j \in S_i\}$. Then
$|S'_i| \le 2 |S_i| \le 12t$ and $\chi(S_i)$ is independent of $\chi(x)$ for all $x \notin S'_i$. So, if $E_i$ depends on $E_j$, it must be the case that $S_j$ intersects $S'_i$.
However, as each $x \in S'_i$ is contained in $t$ sets, there are at most $12 t^2$ such events $E_j$.

We are now in a position to apply the Lov{\'a}sz Local Lemma (Theorem~\ref{thm:LLL}).
Its condition are satisfied as we have $p=1/100 t^2$ and $d=12 t^2$.
Hence $\Pr[\wedge \overline{E_i}]>0$, that is, there exists a coloring $\chi$ consistent with $M$ for which $|\chi(S_i)| \le c \sqrt{t \log t}$ for all small sets $S_i$. This coloring shows that $\disc(\Sigma') \le \max(c \sqrt{t \log t},2)$ as claimed.

\section{Proof of Lemma~\ref{lemma:decompose}}
\label{sec:proof_lemma_decompose}

Let $(X,\Sigma)$ be a $t$-sparse set system with $|X|=n,|\Sigma|=m$. It will
be convenient to identify it with a bi-partite graph $G=(X,V,E)$ where $|V|=m$ and $E=\{(x,i): x \in S_i\}$. Then, a random $t$-sparse set system
is the same as a random left $t$-regular bi-partite graph. That is, a uniform graph satisfying $\deg(x)=t$ for all $x \in X$.

Large sets in $\Sigma$ correspond to the subset of the vertices $V'=\{v \in V: \deg(v) \ge 6t\}$.
For a vertex $v \in V$ let $\Gamma(v) \subset X$ denote its neighbors.
Lemma~\ref{lemma:decompose} is equivalent to the following lemma, which we prove in this section.

\begin{lemma}
\label{lemma:decompose2}
Fix $t \ge 55$. With probability at least $1-2^{-t}$ over the choice of $G$, there exists a choice of $x_v \in \Gamma(v)$ such that
the sets $\{\Gamma(v) \setminus \{x_v\}: v \in V'\}$ are pairwise disjoint.
\end{lemma}

Let $G'$ be the induced (bi-partite) sub-graph on $(X,V')$.
We will show that with high probability $G'$ has no cycles.
In such a case Lemma~\ref{lemma:decompose2} follows from the straightforward scheme
described below:

\begin{claim}
  \label{claim:nocycle_decompose}
  Assume that $G'$ has no cycles. Then there exists a choice of $x_v \in \Gamma(v)$ such that
  the sets $\{\Gamma(v) \setminus \{x_v\}: v \in V'\}$ are pairwise disjoint.
%the conclusion of Lemma~\ref{lemma:decompose2} is true for $G$.
\end{claim}

\begin{proof}
We present a charging scheme of the vertices $x_v \in \Gamma(v)$, for each $v \in V'$.
If $G'$ has no cycles then it is a forest. Fix a tree $T$ in $G'$ and an arbitrary root $v_T \in V'$ of $T$.
Orient the edges of $T$ from $v_T$ to the leaves.
For each $v \in T$ other than the root, choose $x_v$ to be the parent of $v$ in the tree, and choose $x_{v_T}$ arbitrarily.
Let $A_v = \Gamma(v) \setminus \{x_v\}$ for $v \in V'$.
We claim that $\{A_v: v \in V'\}$ are pairwise disjoint. To see that, assume towards contradiction that $x \in A_{v_1} \cap A_{v_2}$ for some $x \in X, v_1,v_2 \in V'$.
Then $v_1,x,v_2$ is a path in $G'$ and hence $v_1,v_2$ must belong to the same tree $T$. However,
the only case where this can happen (as $T$ is a tree) is that $x$ is the parent of both $v_1,v_2$ in $T$.
However, by construction in this case $x=x_{v_1}=x_{v_2}$ and hence $x \notin A_{v_1}, A_{v_2}$,
from which we conclude that $\{A_v: v \in V'\}$ are pairwise disjoint, as claimed.
\end{proof}

In the remainder of the proof we show that with high probability $G'$ has no cycles.
%
%We first show that $V'$ is small.
%\begin{claim}
%\label{claim:prob_V'}
%For any $v \in V$, $\Pr[v \in V'] \le \exp(-2t)$.
%\end{claim}
%
%\begin{proof}
%Fix $v \in V$. For each $x \in X$, $\Pr[x \in S_v] = t/m$. Hence, the degree of $v$ is the sum of $|X|=n$ independent $\{0,1\}$ variables with expectation $t/m$.
%So $\E[\deg(v)] = nt/m \le t$ and hence by Lemma~\ref{lemma:tail2},
%$$
%\Pr[\deg(v) \ge 6t] \le \exp(- 6(m/n)^2 / 3 \cdot (nt/m)) \le \exp(-2t).
%$$
%\end{proof}
%
The girth of $G'$, denoted $\girth(G')$, is the minimal length of a cycle in $G'$ if such exists, and otherwise it is $\infty$.
Note that as $G'$ is bipartite, then $\girth(G')$ is (if finite) the minimal $2\ell$ such that there exist a cycle $x_1,v_1,x_2,v_2,\ldots,x_{\ell},v_{\ell},x_1$ in $G'$ with $x_i \in X$ and $v_i \in V'$.

\begin{claim}
\label{claim:prob_E2}
$\Pr[\girth(G')=4] \le t^4 \exp(-t)$.
\end{claim}

\begin{proof}
Fix $x_1,x_2 \in X$ and $v_1,v_2 \in V$. They form a cycle of length $4$ if $v_1,v_2 \in \Gamma(x_1) \cap \Gamma(x_2)$. As each $\Gamma(x_i)$
is a uniformly chosen set of size $t$ we have that
$$
\Pr[v_1,v_2 \in \Gamma(x_1) \cap \Gamma(x_2)] = \left(\frac{{t \choose 2}}{{m \choose 2}}\right)^2 \le (t/m)^4.
$$
Next, conditioned on the event that $v_1,v_2 \in \Gamma(x_1) \cap \Gamma(x_2)$, we still need to have $v_1,v_2 \in V'$ (that is $v_1$, $v_2$ represent
large sets of $\Sigma$).
We will only require that $v_1 \in V'$
for the bound. Note that so far we only fixed $\Gamma(x_1),\Gamma(x_2)$, and hence the neighbors of $\Gamma(x)$ for $x \ne x_1,x_2$ are still uniform.
Then $v_1 \in V'$ if at least $6t-2$ other nodes $x \in X$ have $v_1$ as their neighbor. By Lemma~\ref{lemma:tail2}, the probability for this is bounded by
$$
\Pr[v_1 \in V' | v_1,v_2 \in \Gamma(x_1) \cap \Gamma(x_2)] \le \exp(-((5t-2)/t)^2 /3 \cdot t) \le \exp(-t).
$$
So,
$$
\Pr[v_1,v_2 \in \Gamma(x_1) \cap \Gamma(x_2) \wedge v_1 \in V'] \le (t/m)^4 \cdot \exp(-t).
$$
To bound $\Pr[\girth(G')=4]$ we union bound over all ${n \choose 2} {m \choose 2}$ choices of $x_1,x_2,v_1,v_2$.
Using our assumption that $m \ge n$ we get
$$
\Pr[\girth(G')=4] \le m^4 (t/m)^4 \exp(-t) \le t^4 \exp(-t).
$$
\end{proof}

\begin{claim}
\label{claim:prob_Ek}
For any $\ell \ge 3$, $\Pr[\girth(G')=2 \ell] \le \exp(-t \ell)$.
\end{claim}

\begin{proof}
Let $x_1,v_1,\ldots,x_{\ell},v_{\ell}$ denote a potential cycle of length $2 \ell$. As it is a minimal cycle and $\ell \ge 3$, the vertices $v_i,v_j$ have no common
neighbors, unless $j=i+1$ in which case $x_i$ is the only common neighbor of $v_i,v_{i+1}$ (where indices are taken modulo $\ell$). Thus there exist
sets $X_i \subset X$ of size $|X_i|=6t-2$ such that $X_i \subset \Gamma(v_i)$ and $X_1,\ldots,X_{\ell},\{x_1,\ldots,x_{\ell}\}$ are pairwise disjoint.

Let $E(x_1,v_1,\ldots,x_{\ell},v_{\ell},X_1,\ldots,X_{\ell})$ denote the event described above, for a fixed choice of $x_1,v_1,\ldots,x_{\ell},v_{\ell},X_1,\ldots,X_{\ell}$.
The event holds if
\begin{enumerate}
\item $v_i,v_{i+1}$ are neighbors of $x_i$.
\item $v_i$ is a neighbor of all $x \in X_i$.
\end{enumerate}
There are independent events, as $\Gamma(x)$ is independently chosen for each $x \in X$. So
\begin{align*}
&\Pr[E(x_1,v_1,\ldots,x_{\ell},v_{\ell},X_1,\ldots,X_{\ell})] \\
&= \prod_{i=1}^{\ell} \Pr[v_i,v_{i+1} \in \Gamma(x_i)] \cdot \prod_{i=1}^{\ell} \prod_{x \in X_i} \Pr[v_i \in \Gamma(x)]\\
&=\left( \frac{{t \choose 2}}{{m \choose 2}} \right)^{\ell} \cdot \left( \frac{t}{m}\right)^{(6t-2)\ell}
\le \left( \frac{t}{m}\right)^{6 t \ell}.
\end{align*}
To bound $\Pr[\girth(G')=2 \ell]$ we union bound over all choices of $x_1,v_1,\ldots,x_{\ell},v_{\ell},X_1,\ldots,X_{\ell}$.
The number of choices is bounded by
$$
n^{\ell} m^{\ell} {n \choose 6t-2}^{\ell} \le \left(\frac{nm \cdot e^{6t-2} \cdot n^{6t-2}}{(6t-2)^{6t-2}}\right)^{\ell}
\le \left(\frac{(em)^{6t}}{(6t-2)^{6t-2}}\right)^{\ell}.
$$
Thus,
$$
\Pr[\girth(G')=2 \ell] \le \left(\frac{(em)^{6t}}{(6t-2)^{6t-2}}\right)^{\ell} \cdot \left( \frac{t}{m}\right)^{6 t \ell} = \left((6t-2)^2 \left(\frac{et}{6t-2}\right)^{6t} \right)^{\ell}
\le \exp(-t \ell).
$$
\end{proof}

\begin{proof}[Proof of Lemma~\ref{lemma:decompose2}]
Using Claims~\ref{claim:prob_E2} and~\ref{claim:prob_Ek}, the probability that $\girth(G') < \infty$ is bounded by:
% Claim~\ref{claim:prob_E2} and Claim~\ref{claim:prob_Ek} by
$$
\Pr[\girth(G') < \infty] = \sum_{\ell=2}^{\infty} \Pr[\girth(G')=\ell] \le t^4 \exp(-t) + \sum_{\ell=3}^{\infty} \exp(-t \ell) \le 2 t^4 \exp(-t).
$$
For $t \ge 55$, we have that $\Pr[\girth(G') < \infty] \le 2^{-t}$.
\end{proof}

%\paragraph{Algorithmic aspects.}

\section{The regime of large sets}

We next prove Theorem~\ref{thm:disc_many_cols}.
Let $(X,\Sigma)$ be a $t$-sparse set system with $|X|=n, |\Sigma|=m$. In this setting, we consider the case of fixed $m,t$
and $n \to \infty$. Consider its $m \times n$ incidence matrix. The columns are $t$-sparse vectors in $\{0,1\}^m$,
and hence have $N={m \choose t}$ possible values. When $n \gg N$, there will be many repeated columns. We show that in this case,
the discrepancy of the set system is low. Setting notations, let $v_1,\ldots,v_N \in \{0,1\}^m$ be all the possible $t$-sparse
vectors, and let $r_1,\ldots,r_N$ denote their multiplicity in the set system.
Note that they define the set system uniquely (up to permutation of the columns, which does not effect the discrepancy).

Our main result in this section is the following. We will assume throughout that $m$ is large enough and that $4 \le t \le m-4$. We note that if $t \le 3$ or $t \ge m-3$ then result immediately follows from the Beck-Fiala theorem (Theorem \ref{thm:bf}), for any set systems. The first case follows by a direct application, and the second case by first partitioning the columns to pairs and subtracting one vector from the next in each pair, which gives a $6$-sparse $\{-1,0,1\}$ matrix, to which we apply the Beck-Fiala theorem.

\begin{theorem}
\label{thm:n_large}
Let $(X,\Sigma)$ be a $t$-sparse set system with $4 \le t \le m-4$ and $m$ large enough.
Assume that $\min(r_1,\ldots,r_N) \ge 7$. Then $\disc(\Sigma) \le 2$.
\end{theorem}

Note that the statement in Theorem~\ref{thm:n_large} is somewhat stronger than that in Theorem~\ref{thm:disc_many_cols},
as it only assumes that all possible $t$-sparse column vectors comprise the incidence matrix of $(X,\Sigma)$,
and their multiplicity is $7$ or higher. In fact, Theorem~\ref{thm:disc_many_cols} follows
from Theorem~\ref{thm:n_large} using a straightforward coupon-collector argument~\cite{ER-61}. In this regime,
with high probability (say, with probability at least $1 - 1/N$), a random sample of $\Theta(N \log{N})$ columns
guarantees that each $t$-sparse column appears with multiplicity $7$ (or higher).
Therefore, we obtain:
$$
\E[\disc(\Sigma)] \le 2\left(1 - \frac{1}{N}\right) + \frac{2t-1}{N}  = O(1) .
$$

We are thus left to prove Theorem~\ref{thm:n_large}.
First, we present an overview of the proof.

\paragraph{Proof overview.}
Every column $v_i$ is repeated $r_i$ times. As we may choose arbitrary signs for each occurrence of a vector, the aggregate total would be $c_i v_i$,
where $c_i \in \Z$, $|c_i| \le r_i$ and $c_i \equiv r_i \mod 2$. Our goal is to show that such a solution $c_i$ always exists, for which
$\|\sum c_i v_i\|_{\infty}$ is bounded, for any initial settings of $r_1,\ldots,r_N$, as long as they are all large enough.

We show that such a solution always exists, with $|c_i| \le 7$. In order to show it, we first fix some solution with the correct parity, and then
correct it to a low discrepancy solution, by adding an even number of copies of each vector.
In order to do that, we study the integer lattice $\L$ spanned by the vectors $v_1,\ldots,v_N$, as our correction comes from $2 \L$.
We show that $\L=\{x \in \Z^m: \sum x_i=0 \mod t\}$, which was already proved by Wilson~\cite{Wilson-90} in a more general scenario.
However, we need an additional property: vectors in $\L$ are efficiently spanned by $v_1,\ldots,v_N$. This allows us to perform the above
correction efficiently, keeping the number of times that each $v_i$ is repeated bounded. Putting that together,
we obtain the result.

\subsection{Proof of Theorem~\ref{thm:n_large}}

Initially, we investigate the lattice spanned by the vectors $v_1,\ldots,v_N$.
As the sum of the coordinates of each of them is $t$, they sit within the lattice
$$
\L = \left\{x \in \Z^m: \sum x_i \equiv 0 \mod t\right\}.
$$
We first show that they span this lattice, and moreover, they do so effectively.

\begin{lemma}
  \label{lemma:lattice}
  For any $w \in \L$ there exist $a_1,\ldots,a_N \in \Z$ such that $\sum a_i v_i =w$. Moreover, $|a_i| \le A$ for all $i \in [N]$ where
  $A=\frac{2 \|w\|_1}{{m-2 \choose t-1}}+2$.
\end{lemma}

\begin{proof}
Assume first that we have $\sum w_i=0$. We will later show how to reduce to this case.
Pair the positive and negative coordinates of $w$. For $L=\|w\|_1 / 2$ let $(i_1,j_1),\ldots,(i_L,j_L)$ be pairs
of elements of $[N]$ such that: if $(i,j)$ is a pair then $w_i>0, w_j<0$; each $i \in [m]$ with $w_i>0$ appears $w_i$ times
as the first element in a pair; and each $j \in [m]$ with $w_j<0$ appears $-w_j$ times as the second element in a pair.
For any $\ell \in [L]$ choose $S_\ell \subset [m]$ of size $t-1$. Set $I_{\ell}=S_{\ell} \cup \{i_{\ell}\}$ and
$J_{\ell}=S_{\ell} \cup \{j_{\ell}\}$.
Identifying $[N]$ with subsets of $[m]$ of size $t$, we have
$$
w = \sum_{\ell=1}^L v_{I_{\ell}} - v_{J_{\ell}}.
$$
We choose the sets $S_1,\ldots,S_{L}$ to minimize the maximum number of times that each vector from $\{v_1,\ldots,v_N\}$
is repeated in the decomposition. When we
choose $S_{\ell}$, we can choose one of $M={m-2 \choose t-1}$ many choices. There is a choice for $S_{\ell}$ such that both $I_{\ell}$ and $J_{\ell}$
appeared thus far less than $2 \ell / M$ times. Choosing such a set, we maintain the invariant that after choosing $S_1,\ldots,S_{\ell}$, each vector
is repeated at most $2 \ell / M+1$ times. Thus, at the end each vector is repeated at most $2L/M + 1$ times.

In the general case, we have $\sum w_i = st$, where we may assume $s>0$. We
apply the previous argument to $w-(v_{i_1}+\ldots+v_{i_s})$, whose coordinates
sum to zero. We choose $i_1,\ldots,i_s \in [N]$ (potentially with repetitions) so as to minimize the maximum number of times that
each vector participates; this number is $\lceil s / N \rceil \le \|w\|_1/M+1$.
Combining the two estimates, we obtain that at the end each vector is repeated at most $4L/M + 2 = 2\|w\|_1/M + 2$ times.
\end{proof}

\begin{lemma}
  \label{lemma:parity}
  For any $b_1,\ldots,b_N \in \{0,1\}$ there exist $c_1,\ldots,c_N \in \Z$ such that
  \begin{enumerate}
  \item[(i)] $c_i \equiv b_i \mod 2$.
  \item[(ii)] $\|\sum c_i v_i\|_{\infty} \le 2$.
  \item[(iii)] $|c_i| \le 7$ for all $i \in [N]$.
  \end{enumerate}
\end{lemma}

\begin{proof}
As a first step, choose $z_i \in \{-1,0,1\}$ such that $z_i=0$ if $b_i=0$, and $z_i \in \{-1,1\}$ chosen uniformly if $b_i=1$. Let $u=\sum z_i v_i$.
Note that for $j \in [m]$,
if there are $k_j$ indices $i \in [N]$ for which $(v_i)_j=1$ and $b_i=1$, then $\E_z[u_j^2]=k_j$. Thus,
$$
\E_z[\|u\|_2^2] = \sum k_j \le Nt.
$$
Thus, with probability at least $1/2$, $\|u\|_2 \le \sqrt{2 N t}$ and hence $\|u\|_1 \le \sqrt{2 N t m}$. Fix such a $u$.

Next, we choose $w \in \L$ such that $\|u - 2 w\|_{\infty}$ is bounded. If we only wanted that $w \in \Z^m$
we could simply choose $q \in \{0,1\}^m$ with $q_i = u_i \mod 2$ and take $w=(u-q)/2$. In order to guarantee that $w \in \L$,
namely that $\sum w_i=0 \mod t$, we change at most $t$ coordinates in $q$ by adding or subtracting $2$.
Thus, we obtain $q \in \{-2,-1,0,1,2\}^m$ where $q_i \equiv u_i \mod 2$ and set $w=(u-q)/2 \in \L$.
We have $\|u - 2 w\|_{\infty} \le 2$.

Next, we apply Lemma~\ref{lemma:lattice} to $w$. We obtain a decomposition $w=\sum a_i v_i$. This implies that
if we set $c_i =z_i - 2 a_i$ then indeed $c_i \equiv b_i \mod 2$ and $\|\sum c_i v_i\|_{\infty} = \| u - 2w\|_{\infty} \le 2$.
To bound $|c_i|$, note that $\|w\|_1 \le \|u\|_1/2+m$. We have by Lemma~\ref{lemma:lattice} that $|a_i| \le A$ for
$$
A = 2 + \eta \le 3,
$$
where
$$
\eta = 2 \frac{\|w\|_1}{{m-2 \choose t-1}} \le O \left(\frac{\sqrt{mt {m \choose t}}}{{m-2 \choose t-1}} \right) \le
O \left(\frac{m^{3/2}}{{m \choose t}^{1/2}} \right) \le 1 ,
$$
whenever $4 \le t \le m-4$ and $m$ is large enough, as is easily verified by the fact that the last term is a
decreasing function of $m$.
\end{proof}

\begin{proof}[Proof of Theorem~\ref{thm:n_large}]
Assume that $r_1,\ldots,r_N \ge 7$. By Lemma~\ref{lemma:parity}, there exists $c_i \in \Z$ such that $c_i \equiv r_i \mod 2$, $|c_i| \le 7$ and
$\|\sum c_i v_i\|_{\infty} \le 2$.
For each $i \in [N]$, we color $|c_i|$ of the vectors $v_i$ with $\text{sign}(c_i) \in \{-1,+1\}$ and the remaining $r_i-|c_i|$ vectors with alternating $+1$ and $-1$
colors (so that their contribution cancels, since $r_i-|c_i|$ is even). The total coloring produces exactly the vector $\sum c_i v_i$, which as guaranteed has discrepancy
bounded by $2$.
\end{proof}

\paragraph*{Acknowledgments.}
The authors wish to thank Aravind Srinivasan for presenting this problem during
a discussion at the IMA Workshop on the Power of Randomness in Computation.

\bibliographystyle{alpha}
\bibliography{random_BF}

\end{document}